\documentclass[oneside,english]{amsart}
\usepackage[T1]{fontenc}
\usepackage[latin9]{inputenc}
\usepackage{amstext}
\usepackage{amsthm}
\usepackage{amssymb}
\usepackage{stmaryrd}

\makeatletter
\numberwithin{equation}{section}
\numberwithin{figure}{section}
\theoremstyle{plain}
\newtheorem{thm}{\protect\theoremname}
\theoremstyle{plain}
\newtheorem{prop}[thm]{\protect\propositionname}
\theoremstyle{plain}
\newtheorem{lem}[thm]{\protect\lemmaname}
\theoremstyle{remark}
\newtheorem{rem}[thm]{\protect\remarkname}
\theoremstyle{plain}
\newtheorem{cor}[thm]{\protect\corollaryname}
\theoremstyle{definition}
\newtheorem{example}[thm]{\protect\examplename}

\makeatother

\usepackage{babel}
\providecommand{\corollaryname}{Corollary}
\providecommand{\examplename}{Example}
\providecommand{\lemmaname}{Lemma}
\providecommand{\propositionname}{Proposition}
\providecommand{\remarkname}{Remark}
\providecommand{\theoremname}{Theorem}

\begin{document}

\title{On (ultra-) completeness numbers and (pseudo-) paving numbers}

\author{Frédéric Mynard}
\begin{abstract}
We study the completeness and ultracompleteness numbers of a convergence
space. In the case of a completely regular topological space, the
completeness number is countable if and only if the space is \v{C}ech-complete,
and the ultracompleteness number is countable if and only if the space
is ultracomplete. We show that the completeness number of a space
is equal to the pseudopaving number of the upper Kuratowski convergence
on the space of its closed subsets, at $\emptyset$. Similarly, the
ultracocompleteness number of a space is equal to the paving number
of the upper Kuratowski convergence on the space of its closed subsets,
at $\emptyset$.
\end{abstract}

\address{New Jersey City University, 2039 J. F. Kennedy Blvd, Jersey City,
NJ 07305}

\email{fmynard@njcu.edu}



\thanks{This work was supported in part by the Separately Budgeted Research
program of NJCU for the academic years 2017-2018 and 2018-2019.}

\maketitle
\global\long\def\G{\mathcal{G}}
 \global\long\def\F{\mathcal{F}}
 \global\long\def\H{\mathcal{H}}
 \global\long\def\L{\operatorname{L}}
\global\long\def\U{\mathcal{U}}
\global\long\def\W{\mathcal{W}}
 \global\long\def\P{\mathcal{P}}
 \global\long\def\B{\mathcal{B}}
 \global\long\def\A{\mathcal{A}}
\global\long\def\D{\mathcal{D}}
\global\long\def\O{\mathcal{O}}
 \global\long\def\N{\mathcal{N}}
 \global\long\def\X{\mathcal{X}}
 \global\long\def\lm{\mathrm{lim}}
 \global\long\def\then{\Longrightarrow}

\global\long\def\V{\mathcal{V}}
\global\long\def\C{\mathcal{C}}
\global\long\def\adh{\mathrm{\mathrm{adh}}}
\global\long\def\Seq{\mathrm{Seq\,}}
\global\long\def\intr{\mathrm{int}}
\global\long\def\cl{\mathrm{cl}}
\global\long\def\inh{\mathrm{inh}}
\global\long\def\diam{\mathrm{diam}}
\global\long\def\card{\mathrm{card\,}}
\global\long\def\S{\operatorname{S}}
\global\long\def\T{\operatorname{T}}
\global\long\def\I{\operatorname{I}}
\global\long\def\BaseD{\operatorname{B}_{\mathbb{D}}}
\global\long\def\AdhD{\operatorname{A}_{\mathbb{D}}}
\global\long\def\K{\operatorname{K}}
\global\long\def\compl{\operatorname{compl}}
\global\long\def\scompl{\operatorname{ucompl}}

\global\long\def\Epi{\operatorname{Epi}}
\global\long\def\De{\operatorname{D}_{1}}
\global\long\def\st{\operatorname{st}}

\section{Introduction}

E. \v{C}ech generalized the traditional metric notion of completeness
to the topological setting in \cite{cech.bicompact}, calling a topological
space \emph{topologically complete }if it is a $G_{\delta}$-subset
of a compact Hausdorff space. This notion is now usually called \emph{\v{C}ech-completeness}. 

Z. Frolík introduced in \cite{Frolik_G} the notion of a $G_{\delta}$-\emph{space},
that is, a topological space that is $G_{\delta}$ in every Hausdorff
space in which it is dense. He proved that the notion is equivalent
to that of \v{C}ech for completely regular (\footnote{In this paper, we assume that completely regular spaces, also called
\emph{functionally regular }in \cite{DM.book}, are Hausdorff.}) spaces, and that it can be characterized in terms of the existence
of a complete sequence of covers (\footnote{Namely, a sequence of open cover of the space with the property that
any filter with a filter-based formed by open sets that contains an
element of each cover has non-empty adherence.}). Importantly, he generalized the notion to that of $G(\mathfrak{m})$-\emph{space}
for an arbitrary cardinal $\mathfrak{m}$ (\footnote{a topological space that is the intersection of $\mathfrak{m}$ many
open subsets, in every Hausdorff space in which it is dense}). For completely regular spaces, compactness is recovered when $\mathfrak{m}=0$,
local compactness (without compactness) when $\mathfrak{m}=1$, and
\v{C}ech-completeness when $\mathfrak{m}=\aleph_{0}$.

S. Dolecki was the first to realize \cite{D.covers} that such notions
of completeness naturally extend from the setting of topological spaces
to the wider context of convergence spaces, and that completeness
is most naturally studied in that setting. He also introduced a \emph{completeness
number,} which, in the case of a completely regular topological space,
is the least cardinal $\mathfrak{m}$ for which the space is a $G(\mathfrak{m})$-space
in the sense of Frolík. 

One of the main reasons to study topological problems in the larger
context of convergence spaces is the availability of a canonical function
space structure (the so-called \emph{continuous convergence} \cite{Binz,BB.book},
also called \emph{natural convergence} \cite{DM.book}) satisfying
the exponential law. Hence the question arises naturally whether this
new cardinal invariant of a space (completeness number) can be related
to a cardinal invariant of a dual space, where duality is considered
with respect to a function space endowed with the natural convergence.
Dolecki followed this line of investigation with \mbox{\cite[Theorem 10.1]{D.covers}},
relating the completeness number of a convergence space $(X,\xi)$
with the local structure at the empty set of the set of closed subsets
of $\xi$ equipped with the upper Kuratowski convergence -- a space
that appears naturally in various contexts (see, e.g., \cite{DGL.kur})
and can be identified with the space $[\xi,\$]$ of continuous functions
from $(X,\xi)$ to the Sierpinski space $\$$, endowed with the natural
convergence. While the approach was very insightful, it turns out
that the result is not correct as stated.

More specifically, \cite[Theorem 10.1]{D.covers} stated that the
completeness number of a convergence space $(X,\xi)$ is the \emph{paving
number }of $[\xi,\$]$ at $\emptyset$, that is, the minimal number
of filters necessary to determine the convergence at $\emptyset$
(See Section \ref{sec:(Pseudo)paving-number} for a proper definition).
It turns out that the paving number of $[\xi,\$]$ at $\emptyset$
is equal to the \emph{ultracompleteness number }of $(X,\xi)$ (Theorem
\ref{thm:scompldual}). Topological spaces with countable ultracompleteness
number have been called cofinally \v{C}ech-complete \cite{romaguera1998cofinally},
strongly complete \cite{ponomarev1987countable}, and ultracocomplete\emph{
}\cite{bijagoar2001ultracomplete} and have been extensively studied. 

On the other hand, the completeness number of $(X,\xi)$ is equal
to the \emph{pseudopaving number }of $[\xi,\$]$ at $\emptyset$ (Theorem
\ref{thm:compldual}), another natural local invariant for convergence
spaces. 

The key results are in Section \ref{sec:Duality-theorems}. To prepare
for it, after introducing basic conventions and notions of convergence
spaces (Section \ref{sec:Notations-and-conventions}), we introduce
and study the completeness and ultracompleteness numbers (Section
\ref{sec:Completeness-and-ultracompletene}). 

It turns out that our convergence-theoretic approach allows us to
refine \cite[Proposition 2]{garcia1999perfect} of Garcia and Romaguera,
which states that the open continuous image of an ultracomplete topological
space is ultracomplete, in contrast to \v{C}ech-complete spaces.
We show that it is enough for the map to be biquotient. 

We then turn to the pseudopaving and paving numbers of a convergence
(\ref{sec:(Pseudo)paving-number}), before particularizing these notions
to the $\$$-dual $[\xi,\$]$ (Section \ref{sec:-dual-of-a}).

We use notations and terminology consistent with the recent book \cite{DM.book},
to which we refer the reader for a comprehensive treatment of convergence
spaces. We only give the strictly necessary convergence-theoretic
definitions in the next section.

\section{\label{sec:Notations-and-conventions}Notations and conventions}

\subsection{Set-theoretic conventions}

Let $2^{X}$ denote the powerset of $X$. If $\P\subset2^{X}$, let
\[
\P_{c}:=\left\{ X\setminus P:P\in\P\right\} ,
\]
\[
\P^{\cup}:=\left\{ \bigcup_{P\in\A}P:\A\subset\P,\card\A<\infty\right\} \text{ and }\P^{\cap}:=\left\{ \bigcap_{P\in\A}P:\A\subset\P,\card\A<\infty\right\} ,
\]
\[
\P^{\uparrow}:=\left\{ A\subset X:\exists P\in\P,P\subset A\right\} \text{ and }\P^{\downarrow}:=\left\{ A\subset X:\exists P\in\P,A\subset P\right\} .
\]
 If $\mathbb{D}\subset2^{2^{X}}$, we write $\mathbb{D}^{a}=\left\{ \D^{a}:\D\in\mathbb{D}\right\} $,
where $a$ ranges over $\cup$, $\cap$, $\uparrow$ and $\downarrow$. 

If $a:2^{X}\to2^{X}$ and $\A\subset2^{X}$, we write 
\[
a^{\natural}\A=\{a(A):A\in\A\}.
\]

We say that two families $\A$ and $\B$ of subsets of $X$ \emph{mesh},
in symbols $\A\#\B$, if $A\cap B\neq\emptyset$ whenever $A\in\A$
and $B\in\B$. We also write 
\[
\A^{\#}=\{B\subset X:\{B\}\#\A\}.
\]
 A family $\D\subset2^{X}$ of non-empty subsets of $X$ is a \emph{filter
}if $\D=\D^{\cap\uparrow}$ and a family $\P$ of proper subsets of
$X$ is an \emph{ideal }if $\P=\P^{\cup\downarrow}$. Of course, $\D$
is a filter if and only if $\D_{c}$ is an ideal.

We denote by $\mathbb{F}X$ the set of filters on $X$, which we order
by the inclusion order of families of subsets of $X$. Maximal filters
are called \emph{ultrafilters}. We denote by $\mathbb{U}X$ the set
of ultrafilters on $X$ and, given $\F\in\mathbb{F}X$, by $\beta\F$
the set of ultrafilters that are finer than $\F$.

\subsection{Convergences}

A \emph{convergence} $\xi$ on a set $X$ is a relation between filters
on $X$ and points of $X$, where we write 
\[
x\in\lm_{\xi}\F
\]
 if $\F$ and $x$ are $\xi$-related, subject to the conditions that
$x\in\lim_{\xi}\{x\}^{\uparrow}$ for every $x\in X$ and that $\lim_{\xi}\F\subset\lim_{\xi}\G$
whenever $\F\leq\G$. We denote by $|\xi|$ the underlying set of
a convergence $\xi$, that is, if $(X,\xi)$ is a convergence space,
$|\xi|=X$. A function $f:|\xi|\to|\tau|$ is \emph{continuous }if
\[
f(\lm_{\xi}\F)\subset\lm_{\tau}f[\F]
\]
for every $\F\in\mathbb{F}|\xi|$, where $f[\F]=\{B\subset|\tau|:f^{-}(B)\in\F\}$. 

Given two convergences $\xi$ and $\theta$ with the same underlying
set, we say that $\xi$ is \emph{finer than $\theta$} or that $\theta$
is \emph{coarser than $\xi$, }in symbols $\xi\geq\theta$, if $\lim_{\xi}\F\subset\lim_{\theta}\F$
for every filter $\F$.

A topology can be seen as a particular kind of convergence. Indeed,
a topology $\tau$ defines a convergence 
\[
x\in\lm_{\tau}\F\iff\F\geq\N_{\tau}(x),
\]
where $\N_{\tau}(x)$ denotes the neighborhood filter of $x$ for
the topology $\tau$. This convergence completely and uniquely determines
the topology, which can thus be identified with the convergence. Such
a convergence is called \emph{topological.}

A subset $C$ of $|\xi|$ is \emph{closed }if $\lim\F\subset C$ whenever
$C\in\F^{\#}$. The collection of closed sets satisfies the properties
of the collection of closed sets for a topology, which we denote by
$\T\xi.$ This is the finest topological convergence that is coarser
than $\xi$. The map $\T$ is a functor, indeed a concrete reflector.

A convergence is a \emph{pseudotopology }if 
\[
\bigcap_{\U\in\beta\F}\lim\U\subset\lim\F.
\]
Given a convergence $\xi$ there is a finest pseudotopology $\S\xi$
that is coarser than $\xi$, which is called \emph{pseudotopological
modification of $\xi$} and is defined by 
\[
\lm_{\S\xi}\F=\bigcap_{\U\in\beta\F}\lm_{\xi}\U.
\]
The map $\S$ is a concrete reflector and $\S\leq\T$, for every topology
is a pseudotopology.

The \emph{adherence of a filter }$\F$ for a convergence $\xi$ is
\[
\adh_{\xi}\F=\bigcup_{\mathbb{F}|\xi|\ni\H\#\F}\lm_{\xi}\H=\bigcup_{\U\in\beta\F}\lm_{\xi}\U.
\]
Evidently, $\adh_{\S\xi}=\adh_{\xi}$ and thus $\lim_{\S\xi}\F=\bigcap_{\mathbb{F}|\xi|\ni\H\#\F}\adh_{\xi}\H$.

A subset $A$ of $|\xi|$ is called \emph{compactoid }if $\lim_{\xi}\U\neq\emptyset$
whenever $A\in\U$ and $\U\in\mathbb{U}|\xi|$, and \emph{compact
}if $\lim_{\xi}\U\cap A\neq\emptyset$ whenever $A\in\U$ and $\U\in\mathbb{U}|\xi|$.
The family $\mathcal{K}$ of compactoid subsets (of a given convergence
space) is clearly an ideal, hence the family $\mathcal{K}_{c}$ of
complements is a (possibly degenerate if the convergence is compact)
filter called \emph{cocompactoid filter}.

Given $f:|\xi|\to Y$, there is a finest convergence on $Y$ making
$f$ continuous, which is called \emph{final convergence for $f$
and $\xi$ }and is denoted $f\xi$. Dually, given $f:X\to|\tau|$,
there is a coarsest convergence on $X$ making $f$ continuous, which
is called \emph{initial convergence for $f$ and $\tau$ }and is denoted
$f^{-}\tau$. Note that a map $f:|\xi|\to|\tau|$ is continuous if
and only if $\tau\leq f\xi$ if and only if $f^{-}\tau\leq\xi$. 

An onto map $f:|\xi|\to|\tau|$ is convergence quotient, or \emph{almost
open}, if $\tau\geq f\xi$ so that if $f$ is continuous and almost
open then $\tau=f\xi$. We say that $f$ is \emph{biquotient }if $\tau\geq\S f\xi$
and \emph{quotient }if $\tau\geq\T f\xi$. Of course, if $\xi$ and
$\tau$ are topological, then $f$ is almost open, biquotient or quotient
in this sense if and only if it is in the classical sense of \cite{quest}. 

\section{\label{sec:Completeness-and-ultracompletene}Completeness and ultracompleteness}

\subsection{Basic notions}

A family $\P$ of subsets of a convergence space $(X,\xi)$ is a \emph{$\xi$-cover
}if $\F\cap\P\neq\emptyset$ for every convergent filter $\F$, and
a $\xi$-\emph{pseudocover }if this condition is restricted to convergent
ultrafilters $\F$. 

Given a family $\mathbb{P}$ of covers, a filter $\F$ is $\mathbb{P}$-\emph{Cauchy}
is $\F\cap\P\neq\emptyset$ for every $\P\in\mathbb{P}$ and $\mathbb{P}$-\emph{preCauchy
}if $\F^{\#}\cap\P\neq\emptyset$ for every $\P\in\mathbb{P}$. A
family $\mathbb{P}$ of $\xi$-covers is called \emph{complete }if
every $\mathbb{P}$-Cauchy filter has non-empty $\xi$-adherence,
and \emph{ultracomplete} if every $\mathbb{P}$-preCauchy filter has
non-empty $\xi$-adherence. 

Recall that a completely regular topological space is \v{C}ech-complete
if and only if it has a countable complete collection of (open) covers.
See for instance, \cite{Frolik_G} (\footnote{Frolik asks for every $\mathbb{P}$-Cauchy filter that additionally
has a filter-base composed of open sets to have adherence points.
Whether we omit the filter-base composed of open sets condition or
not results in an equivalent notion in topological spaces.}), \cite[Theorem 3.9.2]{Eng}, \cite{D.covers}. 

On the other hand, completely regular spaces with a countable ultracomplete
collection of open covers have been called \emph{cofinally} \emph{\v{C}ech-complete}
\cite{romaguera1998cofinally}, \emph{strongly complete} \cite{ponomarev1987countable},
and \emph{ultracocomplete }\cite{bijagoar2001ultracomplete}. We refer
the reader to \cite{jardon2016wats} for a good survey on ultracomplete
spaces (\footnote{In the uniform context, the notion goes back to \cite{howes1971completeness}
and \cite{csaszar1975strongly}.}). 

Note that the usual assumption that the space be completely regular
is irrelevant, and the notions at hand can be considered for general
convergence spaces, as they have been in \cite{D.covers}. We call
a convergence space \emph{countably (ultra)complete }if it admits
a countable (ultra)complete collection of covers.

S. Dolecki showed in \cite{D.covers} that completeness can be reformulated
entirely in terms of filters. Ultracompleteness can be characterized
similarly, proceeding along the lines of \cite{D.covers}. Namely,

\begin{prop}
\cite[Proposition 4.2]{D.covers} \label{prop:acompleteidealcover}
A family $\mathbb{P}$ of covers is complete if and only if $\mathbb{P}^{\cup\downarrow}$
is.
\end{prop}

\begin{lem}
\label{lem:preCauchy} A filter is $\mathbb{P}$-preCauchy if and
only if it is $\mathbb{P}^{\cup\downarrow}$-preCauchy.
\end{lem}

\begin{proof}
If $\F$ is $\mathbb{P}$-preCauchy, then $\F^{\#}\cap\P\neq\emptyset$
for every $\P\in\mathbb{P}$, and a fortiori, $\F^{\#}\cap\P^{\cup\downarrow}\neq\emptyset$
because $\P\subset\P^{\cup\downarrow}$. Thus $\F$ is $\mathbb{P}^{\cup\downarrow}$-preCauchy.

Conversely, if $\F$ is $\mathbb{\mathbb{P}^{\cup\downarrow}}$-preCauchy
and $\P\in\mathbb{P}$, then $\P^{\cup\downarrow}\cap\F^{\#}\neq\emptyset$
so that there is $P_{1},\ldots,P_{n}\in\P$ and $C\subset\cup_{i=1}^{n}P_{i}$
with $C\in\F^{\#}$. Then $\cup_{i=1}^{n}P_{i}\in\F^{\#}$ and thus
there is $i\in\{1\ldots n\}$ with $P_{i}\in\F^{\#}$ because $\F^{\#}$
is a filter-grill. Thus $\F$ is $\mathbb{P}$-preCauchy. 
\end{proof}
Moreover,
\begin{prop}
\cite[Theorem 3.1]{D.covers} A family $\P$ is a $\xi$-cover if
and only if 
\[
\adh_{\xi}\P_{c}=\bigcup_{\F\#\P_{c}}\lm_{\xi}\F
\]
is empty.
\end{prop}

Of course, if $\P$ is an ideal cover, then $\P_{c}$ is a filter,
and $\adh_{\xi}\P_{c}$ is the adherence in the usual sense for filters.
We will say that a filter is \emph{non-adherent }if its adherence
is empty. 

In view of Proposition \ref{prop:acompleteidealcover} and Lemma \ref{lem:preCauchy},
we can reformulate the definitions of complete and ultracomplete collections
in terms of collection of filters.

Accordingly, we call a family $\mathbb{D}$ of non-adherent filters
on $(X,\xi)$ \emph{cocomplete }if for every $\G\in\mathbb{F}X$
\begin{equation}
\adh_{\xi}\G=\emptyset\then\exists\D\in\mathbb{D},\,\G\#\D,\label{eq:complete}
\end{equation}
and\emph{ ultracocomplete }if 
\begin{equation}
\adh_{\xi}\G=\emptyset\then\exists\D\in\mathbb{D},\,\G\geq\D.\label{eq:scomplete}
\end{equation}

In view of the previous discussion,
\begin{prop}
\label{prop:cocomplete} The following are equivalent for a collection
$\mathbb{P}$ of $\xi$-covers:
\begin{enumerate}
\item $\mathbb{P}$ is (ultra)complete;
\item $\mathbb{P}^{\cup\downarrow}$ is (ultra)complete;
\item $\mathbb{P}_{*}=\left\{ \P_{c}:\P\in\mathbb{P^{\cup\downarrow}}\right\} $
is an (ultra)cocomplete collection of non-adherent filters.
\end{enumerate}
\end{prop}

The first observation is that any convergence space $(X,\xi)$ admits
an ultracocomplete family of filters, namely,
\[
\mathbb{D}_{\emptyset}:=\{\F\in\mathbb{F}X:\adh_{\xi}\F=\emptyset\},
\]
which is non-empty unless $\xi$ is compact, in which case any family
of filters is ultracocomplete, including the empty family.

Thus we can define the \emph{completeness number $\compl(\xi)$ }and
\emph{ultracompleteness number $\scompl(\xi)$ }of $\xi$, which are
the smallest cardinality of a cocomplete collection of filters, and
of an ultracocomplete collection of filters respectively.
\begin{rem}
One may consider variants of the notions at hand where adherence is
replaced by limit. For instance, Fletcher and Lindgren \cite{FletcherLindgren}
call \emph{strongly \v Cech-complete }a $T_{1}$ topological space
admitting a countable collection $\mathbb{P}$ of open covers so that
every $\mathbb{P}$-Cauchy filter converges. This property, however,
is not equivalent to its counterpart in terms of filters, for using
covers or ideal covers does make a difference here. In fact, no $T_{1}$
convergence with more than one point would admit any collection of
filters that satisfy the corresponding strong cocompleteness condition,
namely that 
\[
\lm_{\xi}\G=\emptyset\then\exists\D\in\mathbb{D},\,\G\#\D.
\]
\end{rem}

\subsection{Interpretation in the space of ultrafilters}

Given a convergence space $(X,\xi)$, let $\mathbb{U}X$ denote the
set of ultrafilters on $X$, endowed with the usual Stone topology
$\beta$ for which the sets
\[
\beta A=\{\U\in\mathbb{U}X:A\in\U\}
\]
where $A$ ranges over the subsets of $X$ form a basis for the topology.
Recall that closed subsets of $\mathbb{U}X$ are of the form
\[
\beta\F=\{\U\in\mathbb{U}X:\U\geq\F\},
\]
where $\F\in\mathbb{F}X$.

Let $\mathbb{U}_{\xi}X$ denote the set of ultrafilters on $X$ that
converge for $\xi$. The following was already observed implicitly
in \cite{Frolik_G}, and explicitly in \cite{D.covers} and \cite{DM.book},
but we include a proof for the sake of comparison with Theorem \ref{thm:scomplUX}.
We call a subset of a topological space a $G_{\kappa}$-\emph{subset}
(resp. $F_{\kappa}$-\emph{subset}) if it is an intersection of $\kappa$
many open sets (resp. a union of $\kappa$ many closed sets).
\begin{thm}
\label{thm:complUX} \cite{D.covers} The following are equivalent:
\begin{enumerate}
\item $\compl(\xi)\leq\kappa$ ;
\item $\mathbb{U}_{\xi}X$ is a $G_{\kappa}$-subset of $(\mathbb{U}X,\beta)$;
\item $\mathbb{U}X\setminus\mathbb{U}_{\xi}X$ is an $F_{\kappa}$-subset
of $(\mathbb{U}X,\beta)$. 
\end{enumerate}
\end{thm}

\begin{proof}
That $(2)\iff(3)$ is obvious. $(1)\then(3)$: Let $\mathbb{D}$ be
a complete family of filters of cardinality at most $\kappa$. Then
each $\D\in\mathbb{D}$ has empty adherence so that $\beta\D\subset\mathbb{U}X\setminus\mathbb{U}_{\xi}X$.
Moreover, if $\U\in\mathbb{U}X\setminus\mathbb{U}_{\xi}X$, then $\lm_{\xi}\U=\adh_{\xi}\U=\emptyset$
so that there is $\D\in\mathbb{D}$ with $\D\#\U$, equivalently,
$\U\in\beta\D$. Thus 
\[
\mathbb{U}X\setminus\mathbb{U}_{\xi}X=\bigcup_{\D\in\mathbb{D}}\beta\D
\]
 is an $F_{\kappa}$-subset of $(\mathbb{U}X,\beta)$. 

$(3)\then(1)$: If $\mathbb{U}X\setminus\mathbb{U}_{\xi}X$ is an
$F_{\kappa}$-subset of $(\mathbb{U}X,\beta)$, then 
\[
\mathbb{U}X\setminus\mathbb{U}_{\xi}X=\bigcup_{\D\in\mathbb{D}}\beta\D
\]
 for a family $\mathbb{D}$ of filters of cardinality $\kappa$. Because
$\beta\D\subset\mathbb{U}X\setminus\mathbb{U}_{\xi}X$, each $\D\in\mathbb{D}$
has empty adherence. If now $\G\in\mathbb{F}X$ with $\adh_{\xi}\G=\emptyset$
then $\beta\G\subset\mathbb{U}X\setminus\mathbb{U}_{\xi}X=\bigcup_{\D\in\mathbb{D}}\beta\D$
so that there is $\D\in\mathbb{D}$ with $\beta\G\cap\beta\D\neq\emptyset$
so that $\G\#\D$. Thus $\mathbb{D}$ is complete.
\end{proof}
Since closed and compact subsets of $(\mathbb{U}X,\beta)$ coincide,
$F_{\sigma}$-subsets and $\sigma$-compact subsets coincide. Hence,
\begin{cor}
\label{cor:countcomplete} A convergence $\xi$ is countably complete
if and only if $\mathbb{U}X\setminus\mathbb{U}_{\xi}X$ is $\sigma$-compact.
\end{cor}

Recall that a $k$-\emph{cover} of a topological space $X$ is a family
$\C\subset\mathbb{P}(X)$ such that for every compact subset $K$
of $X$, there is $C\in\C$ with $K\subset C$. Recall (e.g., \cite{McCoy})
that the $k$-\emph{Arens number $k\frak{a}(X)$ }of $X$ is the smallest
cardinality of a family of compact sets that is a $k$-cover of $X$.
A topological space is \emph{hemicompact }if $k\frak{a}(X)\leq\omega$.
Clearly, these definition make sense for arbitrary convergence spaces
as well. As compactness is absolute, compact subsets of $S\subset\mathbb{U}X$
are of the form $\beta\H$ for some $\H\in\mathbb{F}X$. 

If $A$ is a subset of a topological space $X$, we denote by $\chi(A)$
the \emph{character of $A$ in $X$}, that is, the smallest cardinality
of a filter base of the filter 
\[
\O_{X}(A):=\{U\in\O_{X}:A\subset U\}^{\uparrow}.
\]

\begin{thm}
\label{thm:scomplUX}
\[
\scompl(\xi)=\chi\left(\mathbb{U}_{\xi}X\right)=k\mathfrak{a}(\mathbb{U}X\setminus\mathbb{U}_{\xi}X).
\]
\end{thm}

\begin{proof}
That $\chi\left(\mathbb{U}_{\xi}X\right)=k\mathfrak{a}(\mathbb{U}X\setminus\mathbb{U}_{\xi}X)$
is easily seen.

Let $\mathbb{D}$ be an ultracocomplete collection of non-adherent
filters. Because each $\D\in\mathbb{D}$ is non-adherent, $\bigcup_{\D\in\mathbb{D}}\beta\D\subset\mathbb{U}X\setminus\mathbb{U}_{\xi}X$.
On the other hand, if $K$ is a compact subset of $\mathbb{U}X\setminus\mathbb{U}_{\xi}X$,
there is $\G\in\mathbb{F}X$ with $\beta\G=K\subset\mathbb{U}X\setminus\mathbb{U}_{\xi}X$.
Hence, $\adh_{\xi}\G=\emptyset$, so that, by (\ref{eq:scomplete}),
there is $\D\in\mathbb{D}$ with $\G\geq\D$, equivalently, $\beta\G\subset\beta\D$.
Thus $\{\beta\D:\D\in\mathbb{D}\}$ is a $k$-cover of $\mathbb{U}X\setminus\mathbb{U}_{\xi}X$
composed of compact subsets of $\mathbb{U}X\setminus\mathbb{U}_{\xi}X$.
Thus $k\frak{a}(\mathbb{U}X\setminus\mathbb{U}_{\xi}X)\leq\scompl(\xi)$.

Conversely, a $k$-cover of $\mathbb{U}X\setminus\mathbb{U}_{\xi}X$
composed of compact subsets of $\mathbb{U}X\setminus\mathbb{U}_{\xi}X$
is of the form $\left\{ \beta\D:\D\in\mathbb{D}\right\} $ for some
family $\mathbb{D}$ of filters. Because $\beta\D\subset\mathbb{U}X\setminus\mathbb{U}_{\xi}X$,
each $\D\in\mathbb{D}$ is non-adherent. If now $\G$ is another non-adherent
filter on $(X,\xi)$ then $\beta\G$ is a compact subsets of $\mathbb{U}X\setminus\mathbb{U}_{\xi}X$,
so that there is $\D\in\mathbb{D}$ with $\beta\G\subset\beta\D$,
equivalently, $\D\leq\G$. Thus $\mathbb{D}$ is ultracocomplete.
Hence $\scompl(\xi)\leq k\frak{a}(\mathbb{U}X\setminus\mathbb{U}_{\xi}X)$.
\end{proof}
\begin{cor}
\label{cor:countultracomplete} A convergence $\xi$ is countably
ultracomplete if and only if $\mathbb{U}X\setminus\mathbb{U}_{\xi}X$
is hemicompact.
\end{cor}

\begin{rem}
Corollaries \ref{cor:countcomplete} and \ref{cor:countultracomplete}
are reminiscent of the classical facts that a completely regular space
is \v{C}ech-complete if and only if its remainder in one (equivalently
all) of its compactification is $\sigma$-compact (e.g., \cite[Theorem 3.9.1]{Eng}),
and ultracomplete if and only if its remainder in one (equivalently
all) of its compactifications is hemicompact \cite{bijagoar2001ultracomplete}.
Of course, corollaries \ref{cor:countcomplete} and \ref{cor:countultracomplete}
make sense whether compactifications of $\xi$ in the usual sense
exist or not, so we can see them as generalization of the classical
statements. Let us clarify further the relationship:

As $(\mathbb{U}X,\beta)$ is the \v{C}ech-Stone compactification
of $X$ endowed with the discrete topology, it has the universal property
that any map $f:X\to K$ where $K$ is a compact Hausdorff topological
space has a unique continuous extension $\bar{f}:\mathbb{U}X\to K$
(defined by $\bar{f}(\U)=\lim_{K}f[\U]$), e.g., \cite[Theorem IX.5.11]{DM.book}.
Hence if $X$ is a completely regular space and $Y$ is a (Hausdorff)
compactification of $X$, then the inclusion map $i:X\to Y$ has a
unique continuous extension $\bar{i}:\mathbb{U}X\to Y$ defined by
$\bar{i}(\U)=\lim_{Y}i[\U]$, where $i[\U]=\U^{\uparrow_{Y}}$ is
the filter generated on $Y$ by $\U$. 

As $X$ is dense, $\bar{i}$ is onto, and thus, as a continuous onto
map between compact Hausdorff spaces, it is also perfect. It is clear
that $\bar{i}(\mathbb{U}X\setminus\mathbb{U}_{\xi}X)=Y\setminus X$.
Moreover, as $\bar{i}$ is continuous and perfect, $Y\setminus X$
is $\sigma$-compact (respectively hemicompact), if and only if $\mathbb{U}X\setminus\mathbb{U}_{\xi}X$
is. In other words, the classical statements for completely regular
spaces are simple corollaries of our corollaries \ref{cor:countcomplete}
and \ref{cor:countultracomplete}. More generally,
\end{rem}

\begin{thm}
Let $X$ be a completely regular topological space. 
\begin{enumerate}
\item The following are equivalent:
\begin{enumerate}
\item $\compl(X)\leq\kappa$;
\item $\mathbb{U}_{\xi}X$ is a $G_{\kappa}$-subset of $(\mathbb{U}X,\beta)$;
\item $X$ is a $G_{\kappa}$-subset of some compactification of $X$;
\item $X$ is a $G_{\kappa}$-subset of all compactifications of $X$.
\end{enumerate}
\item The following are equivalent:
\begin{enumerate}
\item $\scompl(X)\leq\kappa$;
\item $\chi\left(\mathbb{U}_{\xi}X\right)\leq\kappa$ in $(\mathbb{U}X,\beta)$;
\item $\chi(X)\leq\kappa$ in some compactification of $X$;
\item $\chi(X)\leq\kappa$ in all compactifications of $X$.
\end{enumerate}
\end{enumerate}
\end{thm}

\subsection{Finite ultracompleteness numbers}

Since $\compl(\xi)=0$ or $\scompl(\xi)=0$ if the empty family of
covers is (ultra)complete, and since every filter is $\emptyset$-Cauchy,
\[
\compl(\xi)=0\iff\scompl(\xi)=0\iff\xi\text{ is compact. }
\]
On the other hand, 
\begin{prop}
\label{prop:finitenumber} The following are equivalent:
\begin{enumerate}
\item $\scompl(\xi)<\infty$
\item $\scompl(\xi)\leq1$;
\item $\compl(\xi)<\infty$
\item $\compl(\xi)\leq1$;
\item The cocompactoid filter is non-adherent (and possibly degenerate if
$\xi$ compact);
\item $\xi$ is locally compactoid. 
\end{enumerate}
\end{prop}

Note that the case of the completeness number, that is, equivalences
(3) to (6), is \cite[Proposition 7.3]{D.covers}.
\begin{proof}
The equivalence between (3), (4), (5) and (6) is established in \cite[Proposition 7.3]{D.covers}.
Moreover, $(2)\then(1)\then(3)$ is clear, for $\compl(\xi)\leq\scompl(\xi)$.
Hence it is enough to show that $(5)\then(2)$. We claim that if the
cocompactoid filter $\mathcal{K}_{c}$ is non-adherent, then the family
$\{\mathcal{K}_{c}\}$ is ultracocomplete. Indeed, if $\F$ is a non-adherent
filter, then it does not mesh any compactoid set, that is, $K^{c}\in\F$
for every $K\in\mathcal{K}$ and thus $\F\geq\mathcal{K}_{c}$.
\end{proof}

\subsection{Basic facts on ultracompleteness numbers}

In view of Proposition \ref{prop:finitenumber}, a locally compact
space is countably ultracomplete, and a countably ultracomplete space
is countably complete. None of these implications can be reversed.
For instance, the set of irrationals with the topology induced from
$\mathbb{R}$ is a non-ultracomplete completely metrizable, hence
countably complete, space \cite[Example 3.2]{bijagoar2001ultracomplete}.
On the other hand, the space $[0,1]\setminus\{\frac{1}{n}:n\in\mathbb{N}\}$
with its natural topology is an ultracomplete metrizable space that
is not locally compact \cite[Example 3.1]{bijagoar2001ultracomplete}.
At any rate, there are countably complete (metrizable topological)
spaces that are not countably ultracocomplete, so that, in view of
corollaries \ref{cor:countcomplete} and \ref{cor:countultracomplete}:
\begin{prop}
\label{prop:subsetUX} There is a set $X$ and a subspace $S$ of
$(\mathbb{U}X\setminus X,\beta)$ which is $\sigma$-compact but not
hemicompact. 
\end{prop}

This observation will turn out to be useful when considering the paving
number of a convergence, and its variants.

Countable (ultra)completeness is obviously hereditary with respect
to closed subsets, but, unlike countable completeness which is hereditary
with respect to $G_{\omega}$-sets among regular convergences, countable
ultracompleteness is not hereditary (among completely regular spaces)
even with respect to open sets \cite[Example 2.3]{buhagiar2002sums}. 

On the other hand, continuous perfect onto maps preserve countable
completeness (e.g., \cite[Theorem 3.9.10]{Eng}) and countable ultracompleteness
\cite[Theorem 1]{garcia1999perfect} in both directions. The latter
extends to the ultracompleteness number. This can be shown, for instance,
with the obvious variant of the proof of \cite[Theorem XI.6.4]{DM.book}(\footnote{See \cite{dolecki2016completeness} for further generalization of
this result.}) stating that if $f:|\xi|\to|\tau|$ is a continuous surjective perfect
map then $\compl(\tau)=\compl(\xi),$ to the effect that :
\begin{prop}
If $f:|\xi|\to|\tau|$ is a continuous surjective perfect map then
\[
\scompl(\tau)=\scompl(\xi).
\]
\end{prop}

On the other hand, while countable completeness is not always preserved
under open continuous image, countable ultracompleteness is \cite[Proposition 2]{garcia1999perfect}.
This, too, extends to ultracompleteness numbers, and more importantly
can be refined to continuous \emph{biquotient }maps, that is, continuous
onto maps $f:|\xi|\to|\tau|$ with $\tau\geq\S(f\xi)$. Every onto
open map is almost open (that is, $\tau\geq f\xi$) and thus biquotient.
See \cite[Section XV.1]{DM.book} for details.
\begin{thm}
\label{prop:biquotientofultra} Let $f:|\xi|\to|\tau|$ be a continuous
biquotient map. Then
\[
\scompl(\tau)\leq\scompl(\xi).
\]
\end{thm}

\begin{lem}
\label{lem:imagecover} If 
\begin{enumerate}
\item $f:|\xi|\to|\tau|$ is almost open and onto, and $\P$ is a $\xi$-cover
then $f[\P]=\{f(P):P\in\P\}$ is a $\tau$-cover;
\item $f:|\xi|\to|\tau|$ is a biquotient map and $\P$ is an ideal $\xi$-cover
then $f[\P]$ is an (ideal) $\tau$-cover.
\end{enumerate}
\end{lem}

\begin{proof}
(1) Under these assumptions $\tau\geq f\xi$. Let $\F$ be a $\tau$-convergent
filter. Then there is a $\xi$-convergent filter $\G$ with $\F\geq f[\G]$.
As $\P$ is a $\xi$-cover, there is $P\in\P\cap\G$, hence $f(P)\in f[\P]\cap\F$.

(2) Let $\F$ be a $\tau$-convergent filter. Then every $\U\in\beta\F$
is $f\xi$-convergent, and by (1) contains $f(P_{\U})$ for some $P_{\U}\in\P$.
Thus, there is a finite subset $S$ of $\beta\F$ (e.g., by \cite[Proposition II.6.5]{DM.book})
such that 
\[
\bigcup_{\U\in S}f(P_{\U})=f\left(\bigcup_{\U\in S}P_{\U}\right)\in\F.
\]
Since $\P$ is an ideal cover, $\bigcup_{\U\in S}P_{\U}\in\P$ and
thus $f[\P]$ is a $\tau$-cover.
\end{proof}
\begin{proof}[Proof of Theorem \ref{prop:biquotientofultra}]
Suppose $\mathbb{P}$ is an ultracomplete collection of $\xi$-covers,
which we can assume to be ideal covers by Proposition \ref{prop:cocomplete}.
Then 
\[
f\left\llbracket \mathbb{P}\right\rrbracket =\left\{ f[\P]:\P\in\mathbb{P}\right\} 
\]
 is a collection of $\tau$-covers by Lemma \ref{lem:imagecover}.
Moreover, if $\F$ is $f\left\llbracket \P\right\rrbracket $-preCauchy,
then for every $\P\in\mathbb{P}$ there is $P_{\P}\in\P$ with $f(P_{\P})\in\F^{\#}$,
equivalently, $P_{\P}\in(f^{-}[\F])^{\#}$. Since $\mathbb{P}$ is
ultracomplete, there is $\adh_{\xi}f^{-}[\F]\neq\emptyset$ and thus,
by continuity of $f$, $\adh_{\tau}\F\neq\emptyset$.
\end{proof}
In contrast, countable ultracompleteness is not preserved under closed
maps \cite[Example 3.3]{bijagoar2001ultracomplete}, hence not under
hereditarily quotient maps.

\section{\label{sec:(Pseudo)paving-number}(Pseudo)paving number}

A family $\mathbb{D}$ of filters on a convergence space $(X,\xi)$
is a \emph{pavement at $x$ }if every $\D\in\mathbb{D}$ converges
to $x$ and, for every filter $\F$ converging to $x$, there is $\D\in\mathbb{D}$
with $\D\leq\F$. The family $\mathbb{D}$ is a \emph{pseudopavement
at $x$ }if every $\D\in\mathbb{D}$ converges to $x$ and, for every
ultrafilter $\U$ converging to $x$, there is $\D\in\mathbb{D}$
with $\U\in\beta\D$.

Let $\mathfrak{p}(\xi,x)$ denote the \emph{paving number} of $\xi$
at $x$, that is, the smallest cardinality of a pavement at $x$ for
$\xi$, and let $\frak{pp}(\xi,x)$ denote the \emph{pseudopaving
number of $\xi$ at $x$}, that is, the smallest cardinality of a
pseudopavement at $x$ for $\xi$.

Let 
\[
\mathbb{U}_{\xi}(x):=\left\{ \U\in\mathbb{U}X:x\in\lm_{\xi}\U\right\} .
\]

\begin{thm}
\label{thm:pseudopavement} Let $(X,\xi)$ be a pseudotopological
space, and let $\mathbb{D}$ be a family of filters on $X$. The following
are equivalent:
\begin{enumerate}
\item $\mathbb{D}$ is a pseudopavement at $x$;
\item $x\in\lim_{\xi}\D$ for every $\D\in\mathbb{D}$, and, for every $\F\in\mathbb{F}X$,
\[
x\in\lm_{\xi}\F\then\exists\D\in\mathbb{D},\;\D\#\F;
\]
\item 
\[
\mathbb{U}_{\xi}(x)=\bigcup_{\D\in\mathbb{D}}\beta\D.
\]
\end{enumerate}
\end{thm}

\begin{proof}
$(1)\then(2)$: If $x\in\lim_{\xi}\F$ and $\U\in\beta\F$, then $x\in\lim_{\xi}\U$,
so that there is $\D\in\mathbb{D}$ with $\D\leq\U$, because $\mathbb{D}$
is a pseudopavement. Hence $\D\#\F$.

$(2)\then(3)$: Because $x\in\lim_{\xi}\D$ for every $\D\in\mathbb{D}$,
\[
\bigcup\beta\D\subset\mathbb{U}_{\xi}(x).
\]
If now $\U\in\mathbb{U}_{\xi}(x)$, then by (2), there is $\D\in\mathbb{D}$
with $\D\#\U$, equivalently, $\U\in\beta\D$. Hence, we have (3).

$(3)\then(1)$: If $\beta\D\subset\mathbb{U}_{\xi}(x)$ and $\xi$
is pseudotopological, $x\in\lim_{\xi}\D$. If $x\in\lim_{\xi}\U$
for $\U\in\mathbb{U}X$, then $\U\in\mathbb{U_{\xi}}(x),$so that
by (3), there is $\D\in\mathbb{D}$ with $\U\in\beta\D$.
\end{proof}
\begin{cor}
\label{cor:ppavingnUx}
\[
\frak{pp}(\xi,x)=\inf\left\{ \kappa\in\mathrm{Ord}:\mathbb{U}_{\xi}(x)\text{ is a }F_{\kappa}\text{-subset of }\mathbb{U}X\right\} .
\]
In particular, $\frak{pp}(\xi,x)\leq\omega$ if and only if $\mathbb{U}_{\xi}(x)$
is $\sigma$-compact.
\end{cor}

Let us return to characterizing the usual paving number. 

\begin{thm}
Let $(X,\xi)$ be a pseudotopological space. A family $\mathbb{D}$
of filters on $X$ is a pavement at $x$ for $\xi$ if and only if
\[
\left\{ \beta\D:\D\in\mathbb{D}\right\} 
\]
is a $k$-cover of $\mathbb{U}_{\xi}(x)$ by compact subsets of $\mathbb{U}_{\xi}(x)$.
\end{thm}

\begin{proof}
Let $\mathbb{D}$ be a family of compact subsets of $\mathbb{U}_{\xi}(x)$
that is a $k$-cover of $\mathbb{U}_{\xi}(x)$. Since closed subsets
of $\mathbb{U}X$ are of the form $\beta\F$ for some filter $\F\in\mathbb{F}X$,
there is $\mathbb{B}\subset\mathbb{F}X$ with $\mathbb{D}=\{\beta\B:\B\in\mathbb{B}\}$.
We claim that $\mathbb{B}$ is a $\xi$-pavement at $x$. Since $\beta\B\subset\mathbb{U}_{\xi}(x)$,
each $\B\in\mathbb{B}$ converges to $x$, because $\xi$ is a pseudotopology.
If $x\in\lim_{\xi}\G$, then $\beta\G$ is a compact subset of $\mathbb{U}_{\xi}(x)$,
so that there is $\B\in\mathbb{B}$ with $\beta\G\subset\beta\B$,
that is, $\B\subset\G$. 

Conversely, if $\mathbb{B}$ is a $\xi$-pavement at $x$, then $\mathbb{D}=\{\beta\B:\B\in\mathbb{B}\}$
is a family of compact subsets of $\mathbb{U}_{\xi}(x)$, and is a
$k$-cover of $\mathbb{U}_{\xi}(x)$. Indeed, compact sets are of
the form $\beta\G$ for some filter $\G$ on $X$, and $\beta\G\subset\mathbb{U}_{\xi}(x)$
means that $x\in\lim_{\xi}\G$ because $\xi$ is a pseudotopology.
Because $\mathbb{B}$ is a pavement, there is $\B\in\mathbb{B}$ with
$\B\leq\G$, equivalently, $\beta\G\subset\beta\B$. 
\end{proof}
\begin{cor}
\label{cor:pavingnUx}
\[
\mathfrak{p}(\xi,x)=k\frak{a}\left(\mathbb{U}_{\xi}(x)\right).
\]
In particular, $\mathfrak{p}(\xi,x)\leq\omega$ if and only if $\mathbb{U}_{\xi}(x)$
is hemicompact. 
\end{cor}

\begin{example}[A pseudotopology of countable pseudopaving number and uncountable
paving number]
 In view of Proposition \ref{prop:subsetUX}, there is a set $X$
and $S\subset\mathbb{U}X\setminus X$ where $S$ is $\sigma$-compact
but not hemicompact. Let $x_{0}\in X$ and let $\xi$ be the pseudotopology
on $X$ defined by $\mathbb{U}_{\xi}(x_{0})=\{\{x_{0}\}^{\uparrow}\}\cup S$
and $\mathbb{U}_{\xi}(x)=\{\{x\}^{\uparrow}\}$ for all $x\in X$,
$x\neq x_{0}$. In view of corollaries \ref{cor:ppavingnUx} and \ref{cor:pavingnUx},
the pseudotopology $\xi$ satisfies 
\[
\mathfrak{pp}(\xi,x_{0})=\omega<\mathfrak{p}(\xi,x_{0}).
\]
\end{example}

Comparing Theorem \ref{thm:complUX} with Corollary \ref{cor:ppavingnUx},
and Theorem \ref{thm:scomplUX} with Corollary \ref{cor:pavingnUx}
suggests a duality between the pseudopaving number on the local side
and the completeness number on the global side, and similarly between
the paving number on the local side and the strong completeness number
on the global side. As we will see, this duality is made explicit
via the $\$$-dual.

\section{\label{sec:-dual-of-a}$\$$-dual of a convergence space}

\global\long\def\rdc{\operatorname{rdc}}

\subsection{$\$$-dual convergence, reduced and erected filters}

The \emph{natural convergence} (\footnote{also often called \emph{continuous convergence}})
$[\xi,\sigma]$ on the set $C(\xi,\sigma)$ of continuous functions
from $\xi$ to $\sigma$ is the coarsest convergence on $C(\xi,\sigma)$
making the \emph{evaluation map} $ev=\left\langle \cdot,\cdot\right\rangle :|\xi|\times C(\xi,\sigma)\to|\sigma|$
(defined by $ev(x,f)=\left\langle x,f\right\rangle =f(x)$) jointly
continuous. See \cite{Binz}, \cite{BB.book}, \cite{DM.book} for
a systematic study of the natural convergence.

We focus here on the case where $\sigma$ is the Sierpi\'nski space
$\$$, where $|\$|=\{0,1\}$ and the open sets are $\emptyset$, $\{1\}$
and $\{0,1\}$. We call \emph{indicator function of }$A\subset X$
the function $\chi_{A}:X\to|\$|$ defined by $\chi_{A}(x)=0$ if and
only if $x\in A$. Then, a function $f:|\xi|\to|\$|$ is continuous
if and only if it is the indicator function of a closed subset of
$\xi$, so that we identify here the underlying set of $[\xi,\$]$
with $\xi$-closed subsets of $|\xi|$. We call $[\xi,\$]$ the $\$$-\emph{dual
of }$\xi$.

If $\A\subset[\xi,\$]$, then $\rdc\A:=\bigcup_{A\in\A}A$ and if
$\G\in\mathbb{F}|[\xi,\$]|$ then 
\[
\rdc^{\natural}\G:=\left\{ \rdc G:G\in\G\right\} ^{\uparrow}
\]
is a (possibly degenerate) filter on $X$, called \emph{reduced filter
of }$\G$. Then
\begin{prop}
\cite{DGL.kur} \label{propConvuK} Let $\xi$ be a convergence. If
$\G\in\mathbb{F}|[\xi,\$]|$ and $A\in|[\xi,\$]|$ then
\[
A\in\lm_{[\xi,\$]}\G\iff\adh_{\xi}\rdc^{\natural}\G\subset A.
\]
\end{prop}

In particular, if $\xi$ is a topology then 
\[
A\in\lm_{[\xi,\$]}\G\iff\bigcap_{G\in\G}\cl_{\xi}\left(\bigcup_{C\in G}C\right)\subset A,
\]
and the convergence $[\xi,\$]$ is then the \emph{upper Kuratowski
convergence }on the $\xi$-closed subsets of $|\xi|$, which has been
extensively studied, e.g., \cite{DGL.kur,Nog.shak,calbrix.alleche,mynard.uKc,DM.uK}.
This convergence can also be seen as the \emph{Scott convergence }on
the complete lattice $(\C_{\xi},\supset)$ of closed subsets of $\xi$
ordered by reversed inclusion (e.g., \cite[p. 132]{contlattices}). 

Let $(X,\xi)$ be a convergence space. If $F\subset X$, let 
\[
e(F)=\left\{ C=\cl_{\xi}C\subset F\right\} 
\]
denote the \emph{erected image of $F$. If }$\F\in\mathbb{F}X$, let
\[
e^{\natural}\F=\left\{ e(F):F\in\F\right\} ^{\uparrow}
\]
denote its \emph{erected filter }on $|[\xi,\$]|$.

It is readily seen (See e.g., \cite[Section 8]{DM.uK}) that if $\F\in\mathbb{F}|\xi|$
and $\G\in\mathbb{F}|[\xi,\$]|$ then
\begin{equation}
\rdc^{\natural}(e^{\natural}\F)\geq\F\label{eq:downofup}
\end{equation}
and 
\begin{equation}
\G\geq e^{\natural}(\rdc^{\natural}\G).\label{eq:upofdown}
\end{equation}

Thus 
\[
\rdc^{\natural}\G=\rdc^{\natural}\left(e^{\natural}(\rdc^{\natural}\G)\right)
\]
 and $[\xi,\$]$ admits at every point a pavement consisting of \emph{saturated
filters}, that is, of filters of the form $e^{\natural}(\rdc^{\natural}\G)$.
In fact, if $\mathbb{P}$ is a pavement of $[\xi,\$]$ at $A$, then
\begin{equation}
\mathbb{P}':=\left\{ e^{\natural}(\rdc^{\natural}\G):\G\in\mathbb{P}\right\} \label{eq:saturatedpavement}
\end{equation}
 is another pavement at $A$, of the same cardinality and composed
of saturated filters.

Note also \cite{mynard.strong} that 
\[
\rdc^{\natural}(e^{\natural}\F)=\intr_{\xi^{*}}^{^{\natural}}\F,
\]
where $\xi^{*}$ is an Alexandroff topology on $|\xi|$ defined by
\[
\cl_{\xi^{*}}B=\left\{ x\in|\xi|:\cl_{\xi}\{x\}\cap B\neq\emptyset\right\} =\bigcup_{b\in B}\cl_{\xi^{*}}\{b\}.
\]
 The dual Alexandroff topology $\xi^{\bullet}$ is defined by 
\[
\cl_{\xi^{\bullet}}B=\bigcup_{b\in B}\cl_{\xi}\{b\}.
\]
A subset of $|\xi|$ is $\xi^{*}$-open if and only if it is $\xi^{\bullet}$-closed.
Moreover,
\begin{equation}
\A\#\cl_{\xi^{\bullet}}^{\natural}\B\iff\cl_{\xi^{*}}^{\natural}\A\#\B.\label{eq:bulletmeshstar}
\end{equation}
We call a convergence $*$-\emph{regular }(resp. $\bullet$-regular)
if $\lim_{\xi}\cl_{\xi^{*}}^{\natural}\F=\lim_{\xi}\F$ (resp. $\lim_{\xi}\cl_{\xi^{\bullet}}^{\natural}\F=\lim_{\xi}\F$
) for every filter $\F$. In view of (\ref{eq:bulletmeshstar}), we
have:
\begin{lem}
\label{lem:adhptreg} Let $\F$ be a filter on $|\xi|$. If $\xi$
is $*$-regular then
\[
\adh_{\xi}\F=\adh_{\xi}\cl_{\xi^{\bullet}}^{\natural}\F
\]
and if $\xi$ is $\bullet$-regular, then 
\[
\adh_{\xi}\F=\adh_{\xi}\cl_{\xi^{*}}^{\natural}\F.
\]
\end{lem}

\begin{proof}
Since $\cl_{\xi^{\bullet}}^{\natural}\F\leq\F,$ $\adh_{\xi}\F\subset\adh_{\xi}\cl_{\xi^{\bullet}}^{\natural}\F$.
On the other hand, $x\in\adh_{\xi}\cl_{\xi^{\bullet}}^{\natural}\F$
if there is $\G\#\cl_{\xi^{\bullet}}^{\natural}\F$, equivalently,
$\cl_{\xi^{*}}^{\natural}\G\#\F$, with $x\in\lim_{\xi}\G$. Since
$\xi$ is $*$-regular, $x\in\lim_{\xi}\cl_{\xi^{*}}^{\natural}\G$
and thus $x\in\adh_{\xi}\F$. The other equality is proved similarly.
\end{proof}
Moreover, 
\[
\cl_{\xi^{\bullet}}^{\natural}\F\leq\F\leq\rdc^{\natural}(e^{\natural}\F)=\intr_{\xi^{*}}^{^{\natural}}\F
\]
and $\F$ is a reduced filter if and only if both inequalities are
equalities. Note that in particular, $\cl_{\xi^{\bullet}}^{\natural}\F$
is a reduced filter and thus 
\[
\cl_{\xi^{\bullet}}^{\natural}\F=\rdc^{\natural}(e^{\natural}(\cl_{\xi^{\bullet}}^{\natural}\F)).
\]
As a consequence,
\begin{equation}
\F\geq\rdc^{\natural}(e^{\natural}(\cl_{\xi^{\bullet}}^{\natural}\F)).\label{eq:Frefinesptclosure}
\end{equation}

In particular, if $\xi$ is $T_{1}$, then $\rdc^{\natural}(e^{\natural}\F)=\F$.

\subsection{Graph induced by a convergence and case of the $\$$-dual}

A convergence $\xi$ determines a directed graph whose set of vertices
is its underlying set $X$, with the relation 

\[
y\to x\iff x\in\lm_{\xi}\{y\}^{\uparrow}.
\]

If the convergence is centered, then $x\to x$, so that there is a
loop at each vertex. If the convergence is topological then $\to$
is a transitive relation, but in general, it is not. Consider the
forward and backward neighborhoods in this graph: 
\[
N^{\to}(y):=\left\{ x:y\to x\right\} =\lm_{\xi}\{y\}^{\uparrow}\text{ and }N^{\leftarrow}(y):=\{x\in|\xi|:x\to y\}.
\]
Let us call a point $r$ of the graph a \emph{root} if $r\to x$ for
all $x\in X$, and an \emph{end }if $x\to r$ for all $x\in X$. Let
$R(\xi)$ and $E(\xi)$ denote the set of roots, and the set of ends,
of the graph induced by $\xi$, respectively. Of course, in many cases,
$R(\xi)=\emptyset$ and/or $E(\xi)=\emptyset$, but not always. For
instance $R([\xi,\$])=\{\emptyset\}$ and $E([\xi,\$])=\{|\xi|\}$. 

Let
\[
y^{\not\leftarrow}:=\{x\in|\xi|:N^{\leftarrow}(y)\cap N{}^{\leftarrow}(x)\subset R(\xi)\},\;y_{\not\rightarrow}=\{x\in:N^{\to}(y)\cap N^{\to}(x)\subset E(\xi)\}.
\]
 Given $A\subset X$, consider (\footnote{We do not introduce the dual notion using $y_{\not\to}$ instead,
only because we do not have any particular use for it here.})
\[
\cl_{\xi^{\dagger}}A=\bigcap_{c\in\bigcap_{a\in A}a^{\not\leftarrow}}c^{\not\leftarrow}.
\]

A \emph{closure operator} $c:2^{X}\to2^{X}$ satisfies $A\subset c(A)$
for all $A\subset X$, and $c(A)\subset c(B)$ whenever $A\subset B$.
It is \emph{grounded} if additionally $c(\emptyset)=\emptyset$, and
\emph{additive }if $c(A\cup B)=c(A)\cup c(B)$.
\begin{prop}
Given a convergence $\xi$, the induced operator $\cl_{\xi^{\dagger}}$
is an idempotent (not necessarily additive) closure operator. It is
grounded if and only if $\xi$ has no root.
\end{prop}

\begin{proof}
For every $A\subset|\xi|$, we have $A\subset\cl_{\xi^{\dagger}}A$
because if $a_{0}\in A$ and $c\in\bigcap_{a\in A}a^{\not\leftarrow}$,
then in particular, $c\in a_{0}^{\not\leftarrow},$ equivalently,
$a_{0}\in c^{\not\leftarrow}$. 

If $A\subset B$, then $\bigcap_{b\in B}b^{\not\leftarrow}\subset\bigcap_{a\in A}a^{\not\leftarrow}$
and thus $\cl_{\xi^{\dagger}}A\subset\cl_{\xi^{\dagger}}B$.

Note that

\begin{equation}
\bigcap_{a\in\cl_{\xi^{\dagger}}A}a^{\not\leftarrow}=\bigcap_{a\in A}a^{\not\leftarrow},\label{eq:auxdagger}
\end{equation}
so that $\cl_{\xi^{\dagger}}$ is idempotent. To see (\ref{eq:auxdagger}),
note that the inclusion $\subset$ is clear because $A\subset\cl_{\xi^{\dagger}}A$.
On the other hand, if $c\in\bigcap_{a\in A}a^{\not\leftarrow}$ and
$x\in\cl_{\xi^{\dagger}}A$, then $x\in c^{\not\leftarrow}$, equivalently,
$c\in x^{\not\leftarrow}$, which shows the reverse inclusion. 

Finally, $\cl_{\xi^{\dagger}}$ may not be grounded if $\xi$ has
roots, for $\cl_{\xi^{\dagger}}\emptyset=R(\xi)$. Indeed, if $x\in\cl_{\xi^{\dagger}}\emptyset$
then $x\in c^{\not\leftarrow}$ for every $c\in X$, so that in particular
$x\in x^{\not\leftarrow}$ and thus $x$ is a root. Clearly, every
root is in $c^{\not\leftarrow}$ for every $c\in X$. 
\end{proof}
We call a convergence $\xi$ $\dagger$-\emph{regular} if $\lim_{\xi}\F=\lim_{\xi}\cl_{\xi^{\dagger}}^{\natural}\F$
for every $\F\in\mathbb{F}|\xi|$. 

Of course, if $\xi$ is $T_{1}$ and has more than one point, then
$N^{\to}(y)=N^{\leftarrow}(y)=\{y\}$ and $E(\xi)=R(\xi)=\emptyset$
for every $y$ so that $y^{\not\leftarrow}=y_{\not\rightarrow}=X\setminus\{y\}$.
Thus 
\[
\bigcap_{a\in A}a^{\not\leftarrow}=\bigcap_{a\in A}\{a\}^{c}=\left(\bigcup_{a\in A}\{a\}\right)^{c}=A^{c}
\]
 and therefore $\cl_{\xi^{\dagger}}A=\bigcap_{c\in A^{c}}c^{\not\leftarrow}=A$.

In the upper Kuratowski convergence, note that 
\[
\lm_{[\xi,\$]}\{C\}^{\uparrow}=\left\{ D\in\C_{\xi}:C\subset D\right\} ,\text{ that is, }C\to D\iff C\subset D.
\]
Hence, $N^{\to}(C)=\left\{ D\in\C_{\xi}:C\subset D\right\} $, $N^{\leftarrow}(C)=\left\{ D\in\C_{\xi}:D\subset C\right\} $,
so that
\[
C^{\not\leftarrow}=\left\{ D\in\C_{\xi}:C\cap D=\emptyset\right\} 
\]
 and 
\[
C_{\not\rightarrow}=\{D\in\C_{\xi}:C\cup D=X\}=\{D\in\C_{\xi}:X\setminus C\subset D\}.
\]
Moreover, given $G\subset\C_{\xi}$,
\[
\cl_{[\xi,\$]^{\dagger}}G=\bigcap_{C\in\bigcap_{D\in G}D^{\not\leftarrow}}C^{\not\leftarrow}=\left\{ F\in\C_{\xi}:(C\cap\rdc G=\emptyset)\then C\cap F=\emptyset\right\} .
\]
In other words, $F\in\cl_{[\xi,\$]^{\dagger}}G$ if and only if $\O(\rdc G)\subset\O(F)$,
if and only if 
\[
F\subset\ker\O(\rdc G)=\left\{ x:\cl_{\xi}\{x\}\cap\rdc G\neq\emptyset\right\} =\cl_{\xi^{*}}\left(\rdc G\right),
\]
that is,
\begin{equation}
\cl_{[\xi,\$]^{\dagger}}G=e\left(\cl_{\xi^{*}}\left(\rdc G\right)\right).\label{eq:Gdagger}
\end{equation}
 A convergence is \emph{reciprocal }if its induced graph is symmetric,
that is, if
\[
x\to y\then y\to x.
\]
 In a reciprocal convergence $\xi$, the two topologies $\xi^{*}$
and $\xi^{\bullet}$ coincide so that $\cl_{\xi^{*}}\left(\rdc G\right)=\rdc G$
because $\rdc G$ is always $\xi^{*}$-open. Thus
\begin{prop}
\label{prop:downupdagger} If $\xi$ is reciprocal, then for every
$G\subset\C_{\xi}$, 
\[
\cl_{[\xi,\$]^{\dagger}}G=e\left(\rdc G\right).
\]
\end{prop}

\section{\label{sec:Duality-theorems}Duality theorems}

Recall (e.g., \cite{DM.uK,DM.book}) that the \emph{epitopological
reflection }$\Epi\xi$ of a convergence $\xi$ is given by the initial
convergence for the point evaluation map $i:|\xi|\to|[[\xi,\$],\$]|$
defined by $i(x)=ev(x,\cdot)$, and for the convergence $[[\xi,\$],\$]$.

Since $[\xi,\$]=[\Epi\xi,\$]$ and $\Epi\xi$ is a always a $*$-regular
pseudotopology with closed limits (e.g., \cite[Proposition XVII.4.2]{DM.book}),
the assumption that $\xi$ be $*$-regular in the two duality theorems
below is natural and not much to ask.
\begin{thm}
\label{thm:scompldual} Let $\xi$ be a $*$-regular convergence.
Then
\[
\scompl(\xi)=\frak{p}([\xi,\$],\emptyset).
\]
\end{thm}

\begin{proof}
Let $\mathbb{D}$ be a strongly cocomplete collection of non-adherent
filter on $|\xi|$. We claim that the collection 
\[
\left\{ e^{\natural}\D:\D\in\mathbb{D}\right\} 
\]
is a pavement of $[\xi,\$${]} at $\emptyset$. That $\emptyset\in\lm_{[\xi,\$]}e^{\natural}\D$
for every $\D\in\mathbb{D}$ follows from (\ref{eq:downofup}) and
$\adh_{\xi}\D=\emptyset$. If now $\emptyset\in\lim_{[\xi,\$]}\G$
then $\adh_{\xi}\rdc^{\natural}\G=\emptyset$, so that there is $\D\in\mathbb{D}$
with $\rdc^{\natural}\G\geq\D$, because $\mathbb{D}$ is strongly
cocomplete. Hence, 
\[
\G\geq e^{\natural}(\rdc^{\natural}\G)\geq e^{\natural}\D
\]
by (\ref{eq:upofdown}). 

Conversely, let $\mathbb{P}$ be a pavement of $[\xi,\$]$ at $\emptyset$.
Let 
\[
\mathbb{D}=\left\{ \rdc^{\natural}\P:\P\in\mathbb{P}\right\} .
\]
Because $\emptyset\in\lm_{[\xi,\$]}\P$, $\adh_{\xi}\rdc^{\natural}\P=\emptyset$
for every $\P\in\mathbb{P}$. If moreover $\F$ is another filter
on $|\xi|$ with $\adh_{\xi}\F=\emptyset$, then $\adh_{\xi}\cl_{\xi^{\bullet}}^{\natural}\F=\emptyset$
by Lemma \ref{lem:adhptreg} and thus $\emptyset\in\lim_{[\xi,\$]}e^{\natural}(\cl_{\xi^{\bullet}}^{\natural}\F)$
by (\ref{eq:downofup}). Therefore, there is $\P\in\mathbb{P}$ with
$e^{\natural}(\cl_{\xi^{\bullet}}^{\natural}\F)\geq\P$, and thus
$\rdc^{\natural}(e^{\natural}(\cl_{\xi^{\bullet}}^{\natural}\F))\geq\rdc^{\natural}\P$.
Since $\F\geq\rdc^{\natural}(e^{\natural}(\cl_{\xi^{\bullet}}^{\natural}\F))$
by (\ref{eq:Frefinesptclosure}), we conclude that $\mathbb{D}$ is
strongly cocomplete.
\end{proof}
We call a family $\mathbb{P}$ of filters converging to $x$ a $\dagger$-\emph{pseudopavement
at} $x$ if for every $\F\in\mathbb{F}X$ with $x\in\lm_{\xi}\F$
there is $\P\in\mathbb{P}$ with $\P\#\cl_{\xi^{\dagger}}^{\natural}\F$.
Let $\frak{pp^{\dagger}}(\xi,x)$ denote the smallest cardinality
of a $\dagger$-pseudopavement at $x$. Of course, every pseudopavement
is a $\dagger$-pseudopavement because $\cl_{\xi^{\dagger}}^{\natural}\F\leq\F$
for every $\F$, so that $\frak{pp}^{\dagger}(\xi,x)\leq\frak{pp}(\xi,x)$.
If $\xi$ is $T_{1}$ then $\cl_{\xi^{\dagger}}A=A$ for all $A$
and thus $\frak{pp}^{\dagger}(\xi,x)=\frak{pp}(\xi,x)$. On the other
hand, without any separation, the inequality may be strict:
\begin{example}[We may have $\frak{pp}^{\dagger}(\xi)=1$ and $\frak{pp}(\xi,x)$
arbitrary large]
 Consider the ultrafilter convergence $\xi$ of the antidiscrete
topology of an infinite set $X$. Then every point of $X$ is a root
(and an end) and thus $\cl_{\xi^{\dagger}}A=X$ for every $A\subset X$.
As a result, any choice of one ultrafilter $\U_{0}$ forms a $\dagger$-pseudopavement
(at any point), for $\lim\U_{0}=X$ and for any other ultrafilter
$\U$, we have $\cl_{\xi^{\dagger}}^{\natural}\U=\{X\}$ so that $\U_{0}\#(\cl_{\xi^{\dagger}}^{\natural}\U)$.
Hence $\frak{pp}^{\dagger}(\xi,x)=1$ at every $x\in X$. On the other
hand, $\frak{pp}(\xi,x)\geq|X|$. Indeed a pseudopavement (at any
point $x_{0}$) is composed of ultrafilters on $X$, and for every
$y\in X$, $x_{0}\in\lim\{y\}^{\uparrow}$ but the only ultrafilter
meshing with $\{y\}^{\uparrow}$ is $\{y\}^{\uparrow}$. Thus, a pseudopavement
needs to contain at least all principal ultrafilters and is thus of
cardinality at least $\card X$.
\end{example}

\begin{thm}
\label{thm:compldual} Let $\xi$ be a $*$-regular convergence. Then
\[
\compl(\xi)=\frak{pp}^{\dagger}([\xi,\$],\emptyset).
\]
\end{thm}

\begin{proof}
Let $\mathbb{D}$ be a cocomplete family of non-adherent filters on
$|\xi|$, and let $\mathbb{P}=\{e^{\natural}\D:\D\in\mathbb{D}\}$.
Each element of $\mathbb{P}$ converges to $\emptyset$ for $[\xi,\$]$.
Let $\G$ be another filter with $\emptyset\in\lim_{[\xi,\$]}\G$,
that is, $\adh_{\xi}\rdc^{\natural}\G=\emptyset$. Since $\mathbb{D}$
is cocomplete, there is $\D\in\mathbb{D}$ with $\D\#\rdc^{\natural}\G$,
so that $(e^{\natural}\D)\#(e^{\natural}(\rdc^{\natural}\G)$. In
view of (\ref{eq:Gdagger}), $e^{\natural}(\rdc^{\natural}\G)\geq\cl_{\xi^{\dagger}}^{\natural}\G$,
and we conclude that $\mathbb{P}$ is a $\dagger$-pseudopavement
of $[\xi,\$]$ at $\emptyset$.

Conversely, let $\mathbb{P}$ be a $\dagger$-pseudopavement of $[\xi,\$]$
at $\emptyset$. Then
\[
\mathbb{D}=\left\{ \cl_{\xi^{\bullet}}^{\natural}(\rdc^{\natural}\P):\P\in\mathbb{P}\right\} 
\]
is a cocomplete collection of non-adherent filters on $|\xi|$. Indeed,
every $\rdc^{\natural}\P\in\mathbb{D}$ is non-adherent because $\emptyset\in\lim_{[\xi,\$]}\P$
and thus every $\cl_{\xi^{\bullet}}^{\natural}(\rdc^{\natural}\P)$
is non-adherent by $*$-regularity, using Lemma \ref{lem:adhptreg}.
Moreover, if $\F$ is a non-adherent filter on $|\xi|$ then $\emptyset\in\lim_{[\xi,\$]}e^{\natural}\F$,
so that there is $\P\in\mathbb{P}$ with $\P\#\cl_{\xi^{\dagger}}^{\natural}\left(e^{\natural}\F\right)$.
Moreover, $\cl_{\xi^{\dagger}}^{\natural}\left(e^{\natural}\F\right)=e^{\natural}\left(\cl_{\xi^{*}}^{\natural}\left(\rdc^{\natural}(e^{\natural}\F)\right)\right)$
by (\ref{eq:Gdagger}). Then $(\rdc^{\natural}\P)\#\cl_{\xi^{*}}^{\natural}\left(\rdc^{\natural}(e^{\natural}\F)\right)$,
equivalently, $\cl_{\xi^{\bullet}}^{\natural}(\rdc^{\natural}\P)\#\rdc^{\natural}(e^{\natural}\F)$.
Since $\rdc^{\natural}(e^{\natural}\F)\geq\F$, we conclude that $\cl_{\xi^{\bullet}}^{\natural}(\rdc^{\natural}\P)\#\F.$
\end{proof}
As a result, we see that \cite[Theorem 10.1]{D.covers} stating 
\[
\tag{false}\compl(\xi)=\frak{p}([\xi,\$],\emptyset)
\]
is erroneous because, as we have seen, there are countably complete
topological spaces that are not countably ultracomplete, so that 
\[
\compl(\xi)=\omega<\scompl(\xi)=\frak{p}([\xi,\$],\emptyset).
\]
Yet the statement is corrected by the pair of theorems \ref{thm:scompldual}
and \ref{thm:compldual}. 

\bibliographystyle{plain}

\end{document}